\documentclass[10pt, reqno]{amsart}
\usepackage{amsfonts, amsthm, amssymb,verbatim,amscd,amsmath}

\usepackage{braket}
\usepackage{array}
\usepackage{mathtools}
\usepackage{verbatim}
\usepackage{caption}
\usepackage{subcaption}
\usepackage{pdfsync}
\usepackage{tabu}
\usepackage{fullpage}
\usepackage{color}
\usepackage[svgnames]{xcolor}
\usepackage{enumerate}
\usepackage{pst-node,pst-text,pst-3d,pstricks}
\usepackage[dvips,all]{xy}
\usepackage{hyperref}
\usepackage{graphics}
\usepackage{tikz}
\usepackage{pst-plot}
\usepackage{cleveref}

\usepackage{tikz-cd}

\usepackage{xypic}
  \input{xy}
  \xyoption{all}

\usetikzlibrary{patterns}
\usetikzlibrary{calc}
\usetikzlibrary{matrix}
\usetikzlibrary{arrows}
\usetikzlibrary{decorations.pathmorphing}
\usepackage{graphicx,supertabular,paralist}
\usetikzlibrary{decorations.pathreplacing, decorations.markings}
\tikzset{->-/.style={decoration={markings,mark=at position #1 with {\arrow{>}}},postaction={decorate}}}
\usepgflibrary{arrows}
\usetikzlibrary{math}
\numberwithin{equation}{section}
\newcommand{\G}{\mathcal{G}}
\newcommand{\DD}{\mathcal{D}}
\newcommand{\D}{\Delta}
\newcommand{\N}{\mathbb{N}}

\DeclareMathOperator{\reg}{reg}

\DeclareMathOperator{\pdim}{pdim}

\newtheorem{thm}{Theorem}[section]

\newtheorem{question}[thm]{Question}

\newtheorem{corollary}[thm]{Corollary}
\newtheorem{theorem}[thm]{Theorem}
\newtheorem{lemma}[thm]{Lemma}

\theoremstyle{definition}
\newtheorem{definition}[thm]{Definition}

\newtheorem{example}[thm]{Example}
\newtheorem{notation}[thm]{Notation}
\newtheorem{remark}[thm]{Remark}

\usepackage{hyperref}

\begin{document}

\title{Betti Numbers of  Weighted Oriented Graphs}

\author{Beata Casiday}
\address{Department of Mathematics, Yale University, New Haven, CT 06520-8283}
\email{bea.r.casiday@gmail.com}
\urladdr{}

\author{Selvi Kara}
\address{Department of Mathematics and Statistics, University of South Alabama,  411 University Blvd North, Mobile, AL 36688-0002, USA}
\email{selvi@southalabama.edu}
\urladdr{}

\begin{abstract} 
Let $\DD$ be a weighted oriented graph and $I(\DD)$ be its edge ideal. In this paper, we investigate the Betti numbers of $I(\DD)$ via upper-Koszul simplicial complexes, Betti splittings and the mapping cone construction. In particular,  we provide recursive formulas for the Betti numbers of edge ideals of several classes of weighted oriented graphs.  We also identify classes of weighted oriented graphs whose edge ideals  have a unique extremal Betti number which allows us to compute the regularity and projective dimension for the identified classes. Furthermore, we characterize the structure of a weighted oriented graph $\DD$ on $n$ vertices such that $\pdim (R/I(\DD))=n$ where $R=k[x_1,\ldots, x_n]$.
 
\end{abstract}

\subjclass[2010]{}

\maketitle


\section{Introduction}

An \emph{oriented graph} is an ordered pair $\DD= (V(\DD), E(\DD))$   with the vertex set $V(\DD)$, the edge set $E(\DD)$ and an underlying  graph $G$ on which each edge is given an orientation. If $e=\{x,y\}$ is an edge in $G$ and $e$ is oriented from $x$ to $y$ in $\DD$, we denote the oriented edge  by $(x,y)$ to reflect the orientation. In contrast to directed graphs, multiple edges or loops are not allowed in oriented graphs. An oriented graph $\DD$ is called \emph{vertex-weighted oriented} (or simply \emph{weighted}) if each vertex is assigned a weight by a function $w : V(\DD) \rightarrow \mathbb{N}^{+}$ called a weight function. For simplicity, we set $w_i = w(x_i)$ for each $x_i\in V(\DD).$ If  the weight value $w_i$ of vertex $x_i$ is one, we say $x_i$ has a trivial weight in $\DD$ and we call $x_i$ a \emph{trivial vertex}. Otherwise, we say $x_i$ has a non-trivial weight in $\DD$ and call $x_i$ a \emph{non-trivial vertex}.

Let $\DD$ be a weighted oriented graph with the vertex set $V(\DD)=\{x_1,\ldots, x_n\}$ and $R=k[x_1,\ldots, x_n]$ be a polynomial ring  over a field $k$. By identifying vertices of $\DD$ with variables in $R$, the  \emph{edge ideal} of a weighted oriented graph $\DD$ is defined as
$$I(\DD)= (x_ix_j^{w_j} ~:~ (x_i,x_j) \in E(\DD)).$$
If all vertices of $\DD$ have trivial weights, then $I(\DD)$ is the edge ideal of the (undirected, unweighted) underlying graph of $\DD$. Edge ideal of undirected, unweighted graphs are studied extensively in the literature (see \cite{banerjee2017regularity}). The minimal generators of $I(\DD)$ display only the weights of target vertices for each edge. If a vertex has only outgoing edges from it, we call it a \emph{source vertex}. Since weights of source vertices do not appear in the minimal generators of $I(\DD)$, we shall assume that $w_i=1$ if $x_i$ is a source vertex. 

One of the known appearances of edge ideals of weighted oriented graphs is in the algebraic coding theory literature (see \cite{carvalho2017projective,martinez2017minimum}). In particular, the ideal $I=(x_ix_j^{w_j} : 1\leq i < j\leq n)$ where $2\leq w_1\leq \cdots \leq w_n$  is the initial ideal of    the vanishing ideal $I(\mathcal{X})$ of a projective nested cartesian set  $\mathcal{X}$ with respect to the lexicographic order (see \cite[Proposition 2.11]{carvalho2017projective}). Projective nested cartesian codes are images of  degree $d$ evaluation maps on $\mathcal{X}$ and these type of evaluation codes are introduced to generalize the classical projective Reed–Muller type codes. Note that the ideal $I$ is the edge ideal of a weighted oriented complete graph on $n$ vertices with the edge set $\{ (x_i,x_j) : 1\leq i<j \leq n\}$ such that $x_1$ is a source vertex and all other vertices are non-trivial.  Algebraic invariants of $I$ is used in detecting ``good" projective nested cartesian codes, for instance, the Castelnuovo-Mumford regularity of $R/I$, denoted by $ \reg (R/I)$,  is a strict upper bound for an ``optimal" degree $d.$

Edge ideals of weighted oriented graphs  are fairly new objects in the combinatorial commutative algebra community and there have been a few papers investigating algebraic properties and invariants of these objects (see \cite{BBKO, gimenez2018symbolic, ha2019edge, pitones2019monomial,zhu2019projective}). There has been an extensive literature on the edge ideals of (unweighted, unoriented) graphs, and one of the reasons for such fruitful outcomes is due to the squarefree nature of edge ideals of graphs. On the other hand, edge ideals of weighted oriented graphs are not squarefree in general, so many of the established squarefree connections such as Hochster's formula and independence complexes are no longer available in studying our objects of interest. Furthermore, descriptions of edge ideals of weighted oriented graphs differ depending on the orientation and positions of non-trivial weights, making it more complicated to obtain general results for this class of ideals.  

Our general goal is to address these issues in the study of Betti numbers of edge ideals of weighted oriented graphs. One of the essential elements of the paper is the upper-Koszul simplicial complex. The following Hochster-like formula given in \cite[Theorem 2.2]{bayer1999extremal}  exploits the structure of upper Koszul-simplicial complexes and allows one to compute Betti numbers of $I(\DD)$ in terms of the holes of upper-Koszul simplicial complexes.
$$\beta_{i,\textbf{b}} (I) = \dim_k \widetilde{H}_{i-1} (K_{\textbf{b}} (I);k)$$
where $K_{\textbf{b}} (I)$ is the upper-Koszul simplicial complex of a monomial ideal $I$ at multidegree $\textbf{b} \in \mathbb{N}^n$.

Current literature on algebraic invariants of edge ideals of weighted oriented graphs is obtained by focusing on specific classes of graphs with a predetermined orientation or weight assumptions. For instance, in \cite{zhu2019projective}, authors investigate the regularity and the projective dimension of edge ideals of weighted cycles and rooted forests with the assumption that all non-source vertices are non-trivial and graphs are naturally oriented (i.e., all edges are oriented in the same direction). A more general case, class of weighted naturally oriented paths and cycles under any weight distribution is studied in \cite{BBKO}.  All the formulas provided in \cite{BBKO, zhu2019projective} heavily depend on the orientation and weight distributions. In order to obtain more general results, we investigate the Betti numbers and provide recursive formulas for these invariants. With this approach, we successfully recover several results from \cite{zhu2019projective} and suggest an explanation for the nature of regularity formulas given in \cite{BBKO}.

If $\DD$ is a weighted oriented graph on $n$ vertices, then the projective dimension of $ R/I(\DD) $ is at most $n$  by  Hilbert's Syzygy Theorem. One of the main results of the paper classifies all weighted oriented graphs such that the upper bound for the projective dimension is tight (\Cref{thm:completeMaxPdim}). We achieve it by concluding that edge ideals of such weighted oriented graphs have a unique extremal Betti number and it occurs at the multidegree $\textbf{b} =(w_1, \ldots, w_n) \in \mathbb{N}^n.$  In addition, we show that  $\beta_{n-1,\textbf{b}}$ is the unique extremal Betti number for edge ideals of classes of weighted oriented complete graphs and weighted rooted graphs  on $n$ vertices where  $\textbf{b} =(w_1, \ldots, w_n) \in \mathbb{N}^n$ (\Cref{thm:rootedGraph} and \Cref{thm:naturalWOC}). As a result, we provide formulas for the regularity and projective dimension of the edge ideals of those classes of weighted oriented graphs. Other results of the paper focuses on providing recursive formulas for the Betti numbers of edge ideals of weighted oriented graphs. In particular,  we provide such formulas for edge ideals of
\begin{itemize}
    \item weighted oriented complete graphs with at least one sink vertex by making use of Betti splittings (\Cref{thm:completeSink}) and
    \item weighted oriented graphs with at least one sink vertex that is also a leaf by employing the mapping cone construction (\Cref{thm:mappingCone}).
\end{itemize}

Our paper is organized as follows. In \Cref{ch1}, we recall the necessary terminology and results which will be used in the paper. \Cref{ch2}, we provide general results by relating algebraic invariants of edge ideals of weighted oriented graphs and their induced weighted oriented subgraphs. We also show that reducing weight of a non-trivial sink vertex (i.e., a vertex with only incoming edges towards it) reduces the regularity by one and keeps projective dimension unchanged (\Cref{cor:sinkReduction}). In \Cref{ch3}, we prove \Cref{thm:completeMaxPdim} and \Cref{thm:rootedGraph}. \Cref{ch4} is devoted to weighted oriented complete graphs on $n$ vertices. In particular, we provide formulas for the regularity and projective dimension when the edges are oriented in the ``natural" way, i.e.,  $(x_i,x_j)$ for any $1\leq i < j \leq n$. We also prove \Cref{thm:completeSink} in this section. In \Cref{ch5}, we use the mapping cone construction and obtain recursive formulas for Betti numbers of weighted rooted graphs in \Cref{thm:mappingCone}. As an application of  \Cref{thm:mappingCone}, we consider weighted oriented paths and make connections with the regularity formula given in \cite{BBKO}. Finally, in \Cref{ch6}, we raise some questions about the behavior of Betti numbers of edge ideals under weight reduction operations.

\section{Preliminaries}\label{ch1}

In this section, we collect the notation and terminology that will be used throughout the paper. 

Let $R=k[x_1,\ldots, x_n]$ be a polynomial ring over a field $k$ and $ M$ be a finitely generated $R$ module. Then the minimal free resolution of $M$ over $R$ is of the form
$$0 \longrightarrow \bigoplus_{j \in \mathbb{Z}} R(-j)^{\beta_{p,j} (M)} \longrightarrow \bigoplus_{j \in \mathbb{Z}} R(-j)^{\beta_{p-1,j} (M)} \longrightarrow  \cdots \longrightarrow  \bigoplus_{j \in \mathbb{Z}} R(-j)^{\beta_{0,j} (M)} \longrightarrow  M \longrightarrow 0 .$$

The exponents $\beta_{i,j} (M)$ are invariants of the module called the Betti numbers of $M$ and these invariants encode all the information about the minimal free resolution of a module. In general, it is difficult to explicitly compute the Betti numbers. A common approach to go around this issue is to investigate coarser invariants of the module associated to Betti numbers. In this paper, we focus on the Castelnuovo-Mumford regularity (or simply, regularity), projective dimension and the extremal Betti numbers of $M=R/I$ where $I$ is a homogeneous ideal of $R$. The \emph{Castelnuovo-Mumford regularity} and the \emph{projective dimension} of $R/I$ are defined as
$$  \reg (R/I) = \max \{ j-i : \beta_{i,j} (R/I) \neq 0\}  $$
and
  $$\pdim (R/I) = \max \{ i : \beta_{i,j} (R/I) \neq 0\}.$$

A Betti number $\beta_{k,l} (R/I) \neq 0$ is called \emph{extremal} if $\beta_{i,j} (R/I) =0$ for all pairs $(i,j) \neq (k,l)$ with $i\geq k$ and $j \geq l.$ In other words, extremal Betti numbers occupy the upper left corner of a block of zeroes in the Betti diagram of $R/I$ in Macaulay 2. The notion of extremal Betti numbers are introduced in \cite{bayer1999extremal} as a refinement of the notion of the regularity and one can read off the regularity and projective dimension from the extremal Betti numbers. Particularly, $R/I$ has the unique extremal Betti number if and only if $\beta_{p,p+r} (R/I) \neq 0$ where $p = \pdim (R/I)$ and $r = \reg (R/I).$

\subsection{Upper-Koszul Simplicial Complexes}

An important connection in the field of combinatorial commutative algebra is the Stanley-Reisner correspondence which allows one to relate a squarefree monomial ideal with a simplicial complex  (see \cite{francisco2014survey}). As part of the Stanley-Reisner theory, one can compute the Betti numbers of a squarefree monomial ideal through dimensions of holes of a simplicial complex through the Hochster's formula. When the ideal $I$ is no longer squarefree, there is no Stanley-Reisner complex associated to $I$ and Hochster’s formula cannot be applied directly. However, Bayer, Charalambous, and Popescu introduced the upper Koszul simplicial complex in \cite{bayer1999extremal} and provided a Hochster-like formula to compute the multigraded Betti numbers of any monomial ideal.

\begin{definition}
Let $I$ be a monomial ideal in $R=k[x_1,\ldots, x_n].$  The \emph{upper-Koszul simplicial complex} of $I$ at multidegree  $\textbf{b} \in \mathbb{N}^n $ is 
\begin{equation}\label{eq:upperKoszul}
    K_{\textbf{b}} (I)= \{ F \subseteq \{x_1,\ldots, x_n\} ~:~ \frac{x^{\textbf{b}}}{x^F} \in I \} 
\end{equation}
where $\displaystyle x^F = \prod_{x_i \in F} x_i.$
\end{definition}

\begin{theorem}\cite{bayer1999extremal}[Theorem 2.2]\label{thm:bettiKoszul}
Given a monomial ideal $I$ in $R=k[x_1,\ldots, x_n]$, the multigraded Betti numbers of $I$ are
$$\beta_{i,\textbf{b}} (I) = \dim_k \widetilde{H}_{i-1} (K_{\textbf{b}} (I);k)$$
where $\textbf{b} \in \mathbb{N}^n$.
\end{theorem}

\begin{remark}\label{rem:GradedBetti}
Let $I$ be monomial ideal in $R=k[x_1,\ldots, x_n]$ and $\textbf{b} \in \mathbb{N}^n $. If $x^\textbf{b}$ is not equal to a least common multiple of some of the minimal generators of $I,$ then $K_{\textbf{b}} (I)$ is a cone over some subcomplex. Therefore, all non-zero Betti numbers of $I$ occurs in $\mathbb{N}^n$-graded degrees $\textbf{b}$ such that $x^\textbf{b}$ is equal to a least common multiple of some minimal generators of $I.$ 
\end{remark}

\subsection{Betti Splitting}

In the study of Betti numbers of a monomial ideal $I$, one natural approach is to break down the ideal $I$ into smaller pieces and express the Betti numbers of $I$ in terms of the Betti numbers of the smaller pieces.  This strategy was first introduced by Eliahou and Kervaire in \cite{eliahou1990minimal} for monomial ideals and studied in more detail by Francisco, H\`a, and Van Tuyl in \cite{francisco2009splittings}.

\begin{definition}\label{def:bettiSplit}
Let $I, J,$ and $K$ be monomial ideals with generating sets $\mathcal{G} (I),\mathcal{G}(J),$ and $\mathcal{G}(K)$ such that $\mathcal{G}(I)$ is the disjoint union of $\mathcal{G}(J)$ and $\mathcal{G}(K).$ Then $I=J+K$ is called a \emph{Betti splitting} if 
$$\beta_{i,j} (R/I) =\beta_{i,j} (R/J) + \beta_{i,j} (R/K) +\beta_{i-1,j} (R/ J \cap K)$$
for all $i,j >1$.
\end{definition}



The following result is quite useful as it provides a sufficient condition for $I=J+K$ to be a Betti splitting by considering just the generators of $I$.

\begin{theorem}\cite{francisco2009splittings}[Corollary 2.7]\label{thm:bettiSplit}
Suppose that $I=J+K$ where $\mathcal{G} (J)$ contains all the generators of $I$ divisible by some variables $x_i$ and $\mathcal{G}(K)$ is a non-empty set containing the remaining generators of $I.$ If $J$ has a linear resolution, then $I=J+K$ is a Betti splitting.
\end{theorem}

\section{General Results}\label{ch2}

In this section, we recall the notion of inducedness for weighted oriented graphs and introduce a related notion called \emph{weight reduced form} of a weighted oriented graph. Weight reduced form of a weighted oriented graph $\DD$ has the same underlying graph as $\DD$ and the weight of one of the non-trivial vertices in $\DD$ is reduced by one. We call the process of obtaining a weight reduced form of $\DD$ a \emph{weight reduction process}. In the main result of this section, we show that the weight reduction of a non-trivial sink vertex reduces the regularity by one but keeps the projective dimension unchanged.

\begin{definition}
Let $\DD$ be a weighted oriented graph with the underlying graph $G$. We say $\DD'$ is an \emph{induced weighted oriented subgraph} of $\DD$ if
\begin{itemize}
    \item the underlying graph of $\DD'$ is an induced subgraph of $G,$
    \item orientation of $\DD'$ is induced from $\DD,$ i.e., $E(\DD') \subseteq E(\DD)$, and
    \item $w_x (\DD') = w_x (\DD)$ for all $x \in V(\DD')$
\end{itemize}
where $w_x(\DD)$ denotes the weight of $x$ in $\DD$ and $w_x(\DD')$ denotes the weight of $x$ in $\DD'$.  We shall use this notation throughout the text while considering weight of a vertex in two different weighted oriented graphs. 
\end{definition}

It is well-known that Betti numbers of edge ideals of induced subgraphs for unweighted unoriented graphs can not exceed that of the original graph (see \cite[Lemma 4.2]{beyarslan2015regularity}). We provide an analog of this well-known result for weighted oriented graphs in the following lemma.

\begin{lemma}\label{lem:induced}
Let $\DD'$ be an induced weighted oriented subgraph of $\DD$. Then, for all $i,j\geq 0$, we have
$$\beta_{i,j} ( I(\DD')) \leq \beta_{i,j} ( I(\DD)).$$
\end{lemma}
\begin{proof}
Since the notion of inducedness for weighted oriented graphs is an extension of inducedness for unweighted, unoriented graphs, we can adopt the proof of \cite[Lemma 4.2]{beyarslan2015regularity} to obtain the inequality.
\end{proof}

\begin{corollary}\label{cor:induced}
Let $\DD'$ be an induced weighted oriented subgraph of $\DD$. Then 
$$\pdim (R/ I(\DD')) \leq \pdim (R/ I(\DD)) \text{ and }  \reg (R/ I(\DD')) \leq \reg ( R/I(\DD)).$$
\end{corollary}

In the study of edge ideals of weighted oriented graphs, we assume that source vertices have trivial weights. In a similar vein, one can ask whether the same treatment can be applied to sink vertices. In  \cite{ha2019edge}, authors assume that sinks vertices have trivial weights along with source vertices  while investigating Cohen-Macaulayness of edge ideals. In the investigation of algebraic invariants such as regularity, reducing the weight of a non-trivial sink vertex changes the regularity. Thus, one needs to carefully consider the effects of weight reduction process on the multigraded Betti numbers of edge ideals and its  invariants associated to Betti numbers. For this purpose, we introduce a new notion called  \emph{weight reduced form} of a given weighted oriented graph.

\begin{definition}
Let $\DD$ be a weighted oriented graph with at least one non-trivial vertex $x$ such that $w_x (\DD)>1.$ We say $\DD'$ is a \emph{weight reduced form} of $\DD$ if
\begin{itemize}
    \item $\DD'$ has the same orientation as $\DD$,
    \item $w_y(\DD') = w_y (\DD)$ for all $y \neq x$, and
    \item $w_x(\DD') = w_x (\DD)-1$.
\end{itemize}
Depending on the context, we may specify the vertex that is used for weight reduction and say $\DD'$ is a weight reduced form of $\DD$ on $x$.
\end{definition}

We use the following notation throughout this section.

\begin{notation}\label{not:decrease}
Let $\mathcal{D}$ be a weighted oriented graph on $n$ vertices. If  $\mathcal{D}$ has at least one sink vertex with a non-trivial weight, say $x_p$, such that $w_p(\mathcal{D}) =w>1$, then in-neighbors of $x_p$, denoted by  $N_{\mathcal{D}}^- (x_p),$ coincide with all of its neighbors, denoted by $N_{\mathcal{D}} (x_p),$ where 
$$N_{\mathcal{D}}^- (x_p) =N_{\mathcal{D}} (x_p)=  \{ x : (x,x_p) \in E(\mathcal{D}) \} \neq \emptyset.$$ 
We can decompose the edge ideal of $\DD$ as
$$I(\DD) = I(\mathcal{D}\setminus x_p) + \underbrace{(xx_p^w :~ x \in N_{\mathcal{D}} (x_p))}_{J}$$
where $\G(I(\mathcal{D}\setminus x_p) ) = \G(I(\DD)) \setminus \G(J)$.

Let $\DD'$ be the weighted reduced form of $\DD$  on $x_p$. Then
$$I(\DD') = I(\mathcal{D}\setminus x_p)  + \underbrace{(xx_p^{w-1} :~ x \in N_{\mathcal{D}} (x_p))}_{J'}.$$

Note that $\DD\setminus x_p$ is an induced weighted oriented subgraph of $\DD$ and $\DD'$.
\end{notation}

 \begin{remark}
 In the light of \Cref{rem:GradedBetti}, one can observe that for all $i>0$ and $\textbf{b} \in \mathbb{N}^n$,
 \begin{align*}
   \beta_{i,\textbf{b}} ~(I(\DD))   &=0 \text{ if } \textbf{b}_p \neq 0,w,\\
   \beta_{i,\textbf{b}} ~(I(\DD')) & =0 \text{ if } \textbf{b}_p \neq 0,w-1,\\
   \beta_{i,\textbf{b}} ~(I(\DD \setminus x_p)) & =0 \text{ if } \textbf{b}_p \neq 0
 \end{align*}
      because the upper-Koszul complexes of these ideals at the corresponding multidegrees are all cones with apex $x_p$. 
For the sake of completion, we include the proof below for one of the ideals. 
 \end{remark}

\begin{lemma}
Let $\textbf{b} \in \N^n$ such that  $ \textbf{b}_p >0$. Then $K_{\textbf{b}} (I(\DD \setminus x_p))$ is a cone with apex $x_p$.  
\end{lemma}
\begin{proof}
Denote the upper-Koszul complex $K_{\textbf{b}} (I(\DD \setminus x_p))$ by $\Delta$. It suffices to show that $F \cup \{x_p\} \in \D$ for all $F\in \D$.  Let $F \in \Delta$ such that $x_p \notin F$.  Then 
$$\displaystyle m:= \frac{x_1^{b_1}  \cdots x_p^{b_p} \cdots x_n^{b_n}}{x^{F }} \in I(\DD \setminus x_p).$$
Thus there exists $e \in \mathcal{G} (I(\DD \setminus x_p))$ such that $m=e m'$ for some $m' \in R$. Since $b_p >0$ and $x_p \notin F$, the monomial $m'$ is divisible by $x_p$. Hence 
$$\displaystyle \frac{m}{x_p}= \frac{x_1^{b_1}  \cdots x_p^{b_p} \cdots x_n^{b_n}}{x^{F \cup \{x_p\} }} \in I(\DD \setminus x_p)$$
implying that $F \cup \{x_p\} \in \D$. Therefore, $\D$ is a cone with apex $x_p$.
\end{proof}

Since only non-zero Betti numbers for the edge ideals of interest can occur at multidegrees $\textbf{b} \in \mathbb{N}^n$ where $\textbf{b}_p \in \{0,w\}$ for $I(\DD)$ and $\textbf{b}_p \in \{0,w-1\}$ for $I(\DD'),$ we consider these two cases separately by finding relations between the corresponding upper-Koszul simplicial complexes.

\begin{lemma}\label{lem:reduce1}
Let $\textbf{b} \in \N^n$ such that  $ \textbf{b}_p =0$. Then
$$\beta_{i,\textbf{b}} (I(\DD)) = \beta_{i,\textbf{b}} (I(\DD')) = \beta_{i,\textbf{b}} (I(\DD \setminus x_p))$$
\end{lemma}
\begin{proof}
It suffices to show 
$$K_{\textbf{b}} (I(\DD))= K_{\textbf{b}} (I(\DD')) = K_{\textbf{b}} (I(\DD \setminus x_p))$$
by \Cref{thm:bettiKoszul}. It is immediate from the chain of inclusions $I(\DD \setminus x_p) \subseteq I(\DD) \subseteq I(\DD')$ that $$ K_{\textbf{b}} (I(\DD \setminus x_p)) \subseteq K_{\textbf{b}} (I(\DD))  \subseteq K_{\textbf{b}} (I(\DD')).$$
Let $F \in K_{\textbf{b}} (I(\DD))$. Then 
$$\displaystyle m= \frac{x_1^{b_1}  \cdots x_p^{b_p} \cdots x_n^{b_n}}{x^{F }} \in I(\DD).$$
Since $b_p =0$, none of the generators of $J$ divide $m$ and we must have $m \in I(\DD \setminus x_p)$.  Thus $F \in K_{\textbf{b}} (I(\DD \setminus x_p)),$ proving the equality $K_{\textbf{b}} (I(\DD)) =K_{\textbf{b}} (I(\DD \setminus x_p)).$ The remaining equality follows from the same arguments. 
\end{proof}

\begin{lemma}\label{lem:reduce2}
Let $\textbf{b}, \textbf{b}' \in \N^n$ such that $\textbf{b}_p=w$,  $\textbf{b}'_p =w-1$ and $\textbf{b}'_i =\textbf{b}_i$ for all $i\neq p$. Then
$$\beta_{i,\textbf{b}} (I(\DD)) = \beta_{i,\textbf{b}'} (I(\DD'))$$
\end{lemma}

\begin{proof}
As in the proof of the previous lemma, it suffices to show $K_{\textbf{b}} (I(\DD)) = K_{\textbf{b}'} (I(\DD'))$ by \Cref{thm:bettiKoszul}. Let $F \in K_{\textbf{b}'} (I(\DD'))$. By the definition of upper-Koszul simplicial complexes, we have $$ m:= \frac{x^{\textbf{b}'}}{x^{F}}  = \frac{x_1^{\textbf{b}_1} \cdots x_p^{w-1}\cdots  x_n^{\textbf{b}_n}}{x^{F}} \in I(\DD').$$ 

If $m \in J'$, then there exists $xx_p^{w-1} \in J'$ such that $m$ is divisible by $xx_p^{w-1}$.  As a result, $mx_p \in J$ and 
$$ mx_p=  \frac{x_1^{\textbf{b}_1} \cdots x_p^{w}\cdots  x_n^{\textbf{b}_n}}{x^{F}} \in J \subseteq I(\DD)$$
which implies that $F \in K_{\textbf{b}} (I(\DD)).$ Suppose $m \notin J'.$ Then $m$ must be contained in $I(\DD \setminus x_p)$, thus $mx_p \in I(\DD \setminus x_p) \subseteq I(\DD)$ and $F \in K_{\textbf{b}} (I(\DD)).$
 
It remains to prove the reverse containment. If $F \in K_{\textbf{b}} (I(\DD))$, we have  
$$ m:  = \frac{x_1^{\textbf{b}_1} \cdots x_p^{w}\cdots  x_n^{\textbf{b}_n}}{x^{F}} \in I(\DD).$$ 

If $m\in J,$ then $m$ must be divisible by some $xx_p^w\in J$ which implies that $x_p \notin F$.  Since $w>1$, we have
$$ \frac{m}{x_p}  = \frac{x_1^{\textbf{b}_1} \cdots x_p^{w-1}\cdots  x_n^{\textbf{b}_n}}{x^{F}} \in J' \subseteq I(\DD').$$ 
Thus $F $ is a face in $K_{\textbf{b}'} (I(\DD')).$ Suppose $m \notin J$. Then $m$ must be contained in $ I(\DD \setminus x_p).$ Equivalently, there exists a minimal generator $e$ of $I(\DD \setminus x_p)$ such that $m =e \overline{m}$ for a monomial $\overline{m} \in R.$ Since $w>1$, the monomial $\overline{m}$ is divisible by $x_p$ while $e$ is not. Thus $m/x_p$ is still contained in $I(\DD \setminus x_p) \subseteq I(\DD')$ and $F \in K_{\textbf{b}'} (I(\DD')).$ Therefore,
$$K_{\textbf{b}} (I(\DD)) = K_{\textbf{b}'} (I(\DD')).$$

 \end{proof}
 
\begin{corollary}\label{cor:Betti}
Let $\DD$ be a weighted oriented graph with a non-trivial sink vertex $x_p$ and let $\DD'$ be the weight reduced form of $\DD$ on $x_p$. 
\begin{enumerate}
    \item[(a)] The $i^{\text{th}}$ total Betti numbers of $I(\DD)$ and $I(\DD')$ are equal for all $i\geq 0$.
    \item[(b)] We have 
    $$\beta_{i,j} (I(\DD)) = \beta_{i,j-1} ~(I(\DD')) -\beta_{i,j-1} ~(I(\DD\setminus x_p)) +\beta_{i,j} ~(I(\DD\setminus x_p))$$
    for all $i>0, j >1$.
\end{enumerate}
\end{corollary}

\begin{proof}
(a) It follows from \Cref{lem:reduce1} and \Cref{lem:reduce2} that
\begin{align*}
   \beta_{i} (I(\DD)) = \sum_{\textbf{b} \in \mathbb{N}^n} \beta_{i,\textbf{b}} (I(\DD))&= \sum_{\textbf{b} \in \mathbb{N}^n, \textbf{b}_p=0} \beta_{i,\textbf{b}} (I(\DD)) + \sum_{\textbf{b} \in \mathbb{N}^n, \textbf{b}_p=w} \beta_{i,\textbf{b}} (I(\DD)) \\
   &= \sum_{\textbf{b} \in \mathbb{N}^n, \textbf{b}_p=0} \beta_{i,\textbf{b}} (I(\DD')) + \sum_{\textbf{b} \in \mathbb{N}^n, \textbf{b}_p=w-1} \beta_{i,\textbf{b}} (I(\DD')) \\
   &= \sum_{\textbf{b} \in \mathbb{N}^n} \beta_{i,\textbf{b}} (I(\DD')) =   \beta_{i} (I(\DD')).
\end{align*}

(b) By making use of the equalities of multigraded Betti numbers obtained in \Cref{lem:reduce1} and \Cref{lem:reduce2}, one can express the Betti numbers of $I(\DD)$ in terms of those of $I(\DD')$ and $I(\DD\setminus x_p)$. In particular, we have
\begin{align*}
    \beta_{i,j} (I(\DD)) = \sum_{\textbf{b} \in \mathbb{N}^n, |\textbf{b}|=j} \beta_{i,\textbf{b}} ~(I(\DD))  
    &=  \sum_{\substack{|\textbf{b}|=j \\ \textbf{b}_p=w}} \beta_{i,\textbf{b}} ~(I(\DD))  +\sum_{\substack{|\textbf{b}|=j \\ \textbf{b}_p=0}} \beta_{i,\textbf{b}} ~(I(\DD)) \\
     &= \sum_{\substack{|\textbf{b}'|=j-1 \\ \textbf{b}'_p=w-1}} \beta_{i,\textbf{b}'} ~(I(\DD'))  +\sum_{\substack{|\textbf{b}|=j \\ \textbf{b}_p=0}} \beta_{i,\textbf{b}} ~(I(\DD\setminus x_p)) \\
        &=  \beta_{i,j-1} ~(I(\DD')) -\beta_{i,j-1} ~(I(\DD\setminus x_p)) +\beta_{i,j} ~(I(\DD\setminus x_p)) \\
\end{align*}
Note that $\DD \setminus x_p$ is an induced weighted oriented subgraph of $\DD'.$ Thus, by \Cref{lem:induced}, we have  $\beta_{i,j} ~(I(\DD')) \geq \beta_{i,j} ~(I(\DD\setminus x_p))$ for all $i,j$. 
\end{proof}

\begin{corollary}\label{cor:sinkReduction}
Let $\DD$ be a weighted oriented graph with a non-trivial sink vertex $x_p$ and let $\DD'$ be the weight reduced form of $\DD$ on $x_p$.  Then
\begin{enumerate}
    \item[(a)] $\pdim (R/I(\DD)) = \pdim (R/I(\DD')),$
    \item[(b)]$\reg (R/I(\DD)) = \reg (R/I(\DD')) +1.$
\end{enumerate}
\end{corollary}

\begin{proof}
Equality of projective dimensions immediately follows from \Cref{cor:Betti} (a). By making use of \Cref{cor:Betti} (b), we have
$$\reg (R/I(\DD)) = \max \{ \reg (R/I(\DD'))+1, \reg (R/I(\DD\setminus x_p))\}.$$

Since $\reg (R/I(\DD') \geq \reg (R/I(\DD \setminus x_p)$ by \Cref{cor:induced}, we obtain the desired equality
$$\reg (R/I(\DD)) = \reg (R/I(\DD')) +1.$$

\end{proof}

\begin{remark}
Let $\DD$ be a weighted oriented graph with non-trivial weights. In general, values of the non-trivial weights are not necessarily equal. In the light of the above corollary, one can assign $w_i=1$ for all non-trivial sink weights as the base case and gradually obtain the multigraded Betti numbers for arbitrary values of the non-trivial sink weights.  
\end{remark}

\section{Algebraic Invariants via Upper-Koszul Simplicial Complexes}\label{ch3}

The main focus of this section is to obtain formulas for the projective dimension and regularity of edge ideals of weighted oriented graphs by exploiting the structure of related upper-Koszul simplicial complexes. Structure of an upper-Koszul simplicial complex heavily rely on the choice of a multidegree. An ``optimal" choice for a multidegree can be achieved by encoding the weights of all vertices in the multidegree. We use the word optimal to emphasize that this particular multidegree can lead us to a unique extremal Betti number which in turn enables us to compute the projective dimension and the regularity.

\begin{theorem}\cite[Hilbert's Syzygy Theorem]{hilbert1890ueber}
Every finitely generated graded module $M$ over the ring $R=k[x_1,\ldots, x_n]$ has a graded free resolution of length $\leq n.$ Hence $\pdim (M) \leq n$.
\end{theorem}

By Hilbert's Syzygy Theorem, $\pdim (R/I) \leq n$ for any homogeneous ideal $I\subseteq R.$ It is well-known that this bound is tight. A famous example for this instance is the graded maximal ideal $\mathfrak{m}=(x_1,\ldots, x_n)$ as the Koszul complex on the variables $x_1,\ldots, x_n$ gives a minimal free resolution of $R/\mathfrak{m}$ of length $n.$  In a recent paper \cite{alesandroni2020monomial}, the class of monomial ideals with the largest projective  dimension are characterized using dominant sets and divisibility conditions.

In the first result of this section, we characterize the structure of all weighted oriented graphs on $n$ vertices such that projective dimension of their edge ideals attain the largest possible value.

\begin{theorem}\label{thm:completeMaxPdim}
Let $\DD$ be a weighted oriented graph on the vertices $V(\DD) = \{x_1,\ldots, x_n\}$. Then
$$\pdim (R/ I(\DD))=n$$
if and only if there is an edge $e=(x_j,x_i)$ oriented towards $x_i$ for each $x_i \in V(\DD)$ such that $x_j$ has a non-trivial weight.

Furthermore, 
$$\pdim (R/ I(\DD))=n \text{ if and only if } \beta_{n,\textbf{b}} (R/I(\DD)) \neq 0$$ where $\textbf{b} = (w_1,\ldots, w_n) \in \mathbb{N}^n.$
\end{theorem}

\begin{proof}
Observe that  $\pdim (R/ I(\DD))=n$ if and only if  there exists a multidegree $\textbf{a} \in \mathbb{N}^n$ such that $\beta_{n,\textbf{a}} (R/I(\DD))$ is non-zero. It follows from  \Cref{thm:bettiKoszul} that 
$$\beta_{n,\textbf{a}} (R/I(\DD)) = \dim_k \widetilde{H}_{n-2} (K_{\textbf{a}} (I(\DD) ;k)) \neq 0$$
which happens only when $F= \{x_1,\ldots, x_n\}$ is a minimal non-face of $K_{\textbf{a}} (I(\DD))$. Note that each $a_i \in\{0, 1,w_i\}$. Otherwise, $\beta_{n,\textbf{a}} (R/I(\DD))=0$ by \Cref{rem:GradedBetti}.

Let $\D =K_{\textbf{a}} (I(\DD))$ and $F_i: = F \setminus \{x_i\}$ for each $x_i \in F$. Recall that $F$ is a minimal non-face of $\D$ whenever $F \notin \D$ and each $F_i \in \D$. It follows from the definition of upper-Koszul simplicial complexes  that 
\begin{align*}
  F \notin \D \iff  m:= \frac{x^{\textbf{a}}}{x^{F}} & =  \prod_{j=1}^n x_j^{a_j-1}  \notin I(\DD), \text{ and }  \\
F_i \in \D \iff  m_i := \frac{x^{\textbf{a}}}{x^{F_i}} & =  x_i^{a_i} \Big( \prod_{j\neq i} x_j^{a_j-1} \Big) \in I(\DD)  \text{ for each } i \in \{1,\ldots, n\}.
\end{align*}
Note that each  $a_i \neq 0$ when $m\notin I(\DD)$ and each $m_i \in I(\DD)$. The monomial $m_i \in I(\DD)$ if and only if there exists a minimal generator of $I(\DD)$ associated to an edge oriented towards $x_i$, say $e=(x_k,x_i)$, such that $x_kx_i^{w_i} \in I(\DD)$ divides $m_i$ for some $x_k \in V(\DD)$. This implies that $a_i=w_i$ and $a_k=w_k> 1$. Therefore, $\pdim (R/ I(\DD))=n$ if and only if,  for each $x_i \in V(\DD)$, there exists an edge $e=(x_k,x_i)$ oriented towards $x_i$ such that $w_k >1$. Notice that there are no source vertices when $\pdim (R/ I(\DD))=n$. The latter statement follows from the conclusion that $a_i=w_i$ for each $i=1,\ldots, n$  whenever  $ \beta_{n,\textbf{a}} (R/I(\DD)) \neq 0.$
\end{proof}

Note that $\beta_{n,\textbf{b}} (R/I(\DD)$ is the unique extremal Betti number of $R/I(\DD)$. Using this information, we can further deduce the formula for the regularity of $R/I(\DD)$.

\begin{corollary}
Let $\DD$ be a weighted oriented graph on the vertices $V(\DD) = \{x_1,\ldots, x_n\}$. Then
$$\pdim (R/ I(\DD))=n \text{ if and only if } \reg (R/ I(\DD))= \sum_{i=1}^n w_i -n.$$
\end{corollary}

\begin{proof}
Observe that $\reg (R/ I(\DD))= \sum_{i=1}^n w_i -n$ if and only if $\beta_{n,\textbf{b}} (R/I(\DD)) \neq 0$ where $\textbf{b} = (w_1,\ldots, w_n) \in \mathbb{N}^n.$
\end{proof}

\begin{remark}\label{rem:extremalBetti}
Let $\DD$ be a weighted oriented graph on $n$ vertices and let $\textbf{b}= (w_1, \ldots, w_n ) \in \mathbb{N}^n$ be a multidegree  corresponding to the least common multiple of all minimal generators of $I(\DD)$. 
If $\beta_{p,\textbf{b} }(R/I(\DD) ) $ is non-zero  where $p =\pdim (R/I(\DD)$, then $\beta_{p,|\textbf{b}| }(R/I(\DD))$ is the unique extremal Betti number of $R/I(\DD)$. Because  $\beta_{i,\textbf{a} }(R/I(\DD) ) = 0$ for  $\textbf{a}= (a_1, \ldots, a_n ) \in \mathbb{N}^n$ such that $a_i > w_i$ for some $i=1,\ldots, n$ by  \Cref{rem:GradedBetti}. 
\end{remark}

\subsection{Weighted Oriented Cycles}

Let $\mathcal{C}_n$ denote a weighted oriented cycle on $n$ vertices $x_1,\ldots, x_n$. We shall assume that there exists at least one vertex $x_i$ such that $w_i >1$ and there is an edge oriented into $x_i$. Otherwise, $\mathcal{C}_n$ can be considered as an unweighted, unoriented cycle whose Betti numbers are computed  in \cite{jacques2004betti}.

If $\mathcal{C}_n$ has at least one sink vertex, using \Cref{cor:Betti} (b), we can express the Betti numbers of $I(C_n)$ recursively in terms of Betti numbers of a weighted oriented path on $(n-1)$ vertices and a weight reduced form of $\mathcal{C}_n$. In the existence of one sink vertex, without loss of generality, we may assume that $x_n$ is a sink by reordering vertices of $\mathcal{C}_n$.

\begin{corollary}
If $x_n$ is a sink in $\mathcal{C}_n$ such that $w_n>1,$ then
    $$\beta_{i,j} (I(\mathcal{C}_n)) = \beta_{i,j-1} ~(I(\mathcal{C}'_n)) -\beta_{i,j-1} ~(I(\mathcal{P}_{n-1})) +\beta_{i,j} ~(I(\mathcal{P}_{n-1}))$$
    for all $i>0,j>1$ where $\mathcal{C}'_n$ is a weight reduced form of $\mathcal{C}_n$.
\end{corollary}

If $\mathcal{C}_n$ has no sink vertices (or source vertices), then  $\mathcal{C}_n$ must be endowed with the  natural orientation (clockwise or counter clockwise).  If all vertices of a naturally oriented weighted cycle $\mathcal{C}_n$ have non-trivial weights, its projective dimension and regularity can be computed by \Cref{thm:completeMaxPdim} and we recover one of the main results of \cite{zhu2019projective}.

\begin{corollary}\cite[Theorem 1.4.]{zhu2019projective}
Let $\mathcal{C}_n$ be naturally oriented weighted cycle where $w_i>1$ for all $i$. Then
$$\pdim (R/I(\mathcal{C}_n))=n \text{ and } \reg (R/I(\mathcal{C}_n))= \sum_{i=1}^n w_i -n.$$
\end{corollary}

\subsection{Weighted  Rooted Graphs} In this subsection, we consider the possibility of having source vertices and their effects on algebraic invariants. In the existence of several source vertices, one needs more information on the structure of a weighted oriented graph to be able to compute its algebraic invariants. As a natural starting point, we consider weighted rooted graphs.

\begin{definition}
 A weighted graph $\DD$ is called \emph{rooted} if there is a vertex distinguished as the root and there is a naturally oriented path (i.e., all edges on the path are in the same direction) from the root vertex to any other vertex in $\DD$. Orientation of $\DD$ is determined by  naturally oriented paths from the root to other vertices. Note that the only source vertex of a weighted  rooted graph is its root.
\end{definition}

\begin{remark}
If $\DD$ is a weighted rooted graph, it does not fit in the description of weighted oriented graphs whose edge ideals have the largest projective dimension in \Cref{thm:completeMaxPdim}. Thus $\pdim (R/ I(\DD)) \leq n-1.$
\end{remark}

\begin{theorem}\label{thm:rootedGraph}
Let $\DD$ be a weighted  rooted graph on the vertices $\{x_1,\ldots, x_n\}$ with the root vertex $x_1.$ Suppose $w_i \geq 2$ for all $i\neq 1.$ Then
$$\pdim (R/I(\DD))=n-1 \text{ and } \reg (R/I(\DD))= \sum_{i=1}^n w_i -n+1.$$
\end{theorem}

\begin{proof} Let $\textbf{b}=(1,w_2,\ldots, w_n) \in \mathbb{N}^n.$ It suffices to show that $\beta_{n-1, |\textbf{b}|} (R/I(\DD))$ is the unique extremal Betti number of $R/I(\DD).$

  Consider the upper-Koszul complex $K_{\textbf{b}} (I(\DD))$ of $I(\DD)$. Let $F=\{x_2,\ldots, x_n\}$ and $F_i= F \setminus \{x_i\}$ for $2\leq i\leq n.$ Our goal is to show that $F$ is a minimal non-face of $K_{\textbf{b}} (I(\DD))$. First observe  that $ F \notin K_{\textbf{b}} (I(\DD)).$ Otherwise, by the definition of  upper-Koszul simplicial complexes, we must have $x_1\prod_{i=2}^n x_i^{w_i-1} \in I(\DD)$ which is not possible because each generator of $I(\DD)$ must be divisible by $x_i^{w_i}$ for some $i\in\{2,\ldots, n\}$. 

 If each $F_i \in K_{\textbf{b}} (I(\DD))$, then $F$ must be a minimal non-face.  Note that 
 $$m_i:= \frac{\prod_{j=1}^n x_j^{w_j}}{ x_2 \cdots \hat{x_i} \cdots x_n} = x_1 \Big( \prod_{i=1, j\neq i}^n x_j^{w_j-1} \Big) x_i^{w_i} \in I(\DD)$$
 because all non-source vertices have non-trivial weights and $m_i$ is divisible by $x_1x_i^{w_i} \in I(\DD)$ or $x_jx_i^{w_i} \in I(\DD)$ for some $i\in\{2,\ldots, n\}.$ It implies that $F_i \in K_{\textbf{b}} (I(\DD))$  for each $i$. Thus  $F$ is an $(n-2)$-dimensional minimal non-face of $K_{\textbf{b}} (I(\DD))$. It follows from \Cref{thm:bettiKoszul} that 
$$\beta_{n-1,\textbf{b}} (R/I(\DD))= \dim_k \widetilde{H}_{n-3} (K_{\textbf{b}} (I(\DD));k)) \neq 0.$$
 
 Therefore, $\beta_{n-1,|\textbf{b}|} (R/I(\DD))$ is the unique extremal Betti number of $R/I(\DD)$  by \Cref{thm:completeMaxPdim} and \Cref{rem:extremalBetti}.
\end{proof}

As an immediate consequence of \Cref{thm:rootedGraph}, we recover several results from \cite{zhu2019projective}. 

\begin{corollary}\cite[Theorem 1.2. and Theorem 1.3.]{zhu2019projective}\label{cor:weightedForest}
Let $\mathcal{D}$ be weighted rooted forest with non-trivial weights. Then
$$\pdim (R/I(\DD))=n-1 \text{ and } \reg (R/I(\DD))= \sum_{i=1}^n w_i -n+1.$$
\end{corollary}

\section{Weighted Oriented Complete Graphs}\label{ch4}

Let $\mathcal{K}_n$ denote a weighted oriented complete graph on $n$ vertices $\{x_1,\ldots, x_n\}$ for $n>1$ and let $I(\mathcal{K}_n)$ denote its edge ideal. Throughout this section, we may assume that there exists at least one vertex $x_p$ such that $w_p >1$. Otherwise, $\mathcal{K}_n$ is the unweighted, unoriented complete graph and  Betti numbers of its edge ideal are well-understood from its  independence complex due to Hochster's formula.

\begin{definition}\label{not:complete}
A weighted oriented complete graph $\mathcal{K}_n$ is called \emph{naturally oriented} if the oriented edge set is given by $\{(x_i ,x_j) : 1\leq i<j\leq n\}$. Then the edge ideal of $\mathcal{K}_n$ is
$$I(\mathcal{K}_n)= (x_ix_j^{w_j} : 1 \leq i < j \leq n).$$
Since $x_1$ is a source vertex, we set $w_1=1$.
\end{definition}

In what follows, we provide formulas for the projective dimension and the regularity for the edge ideal of a naturally oriented weighted complete graph. The  key ingredient of the proof is the use of upper-Koszul simplicial complex of $I(\mathcal{K}_n)$.

\begin{theorem}\label{thm:naturalWOC}
Let $\mathcal{K}_n$ be a naturally oriented weighted  complete graph such that $w_p>1$ for some $p \geq 2$. Then
\begin{enumerate}
     \item[(a)] $\pdim (R/I(\mathcal{K}_n))=n-1$ and
     \item[(b)] $\displaystyle \reg (R/I(\mathcal{K}_n))= \sum_{i=1}^n w_i -n+1.$
\end{enumerate}
\end{theorem}

\begin{proof}
Let $\textbf{b}=(w_1,w_2,\ldots, w_n) \in \mathbb{N}^n$. We claim that $\beta_{n-1, \textbf{b}} (R/I(\mathcal{K}_n))$ is the unique extremal Betti  number of $R/I(\mathcal{K}_n)$. As an immediate consequence of the claim, we obtain the expressions given in the statement of the theorem as the values of the projective dimension and  the regularity of $R/I(\mathcal{K}_n)$.

In order to prove the claim, consider the upper Koszul simplicial complex of $I(\mathcal{K}_n)$ in multidegree $\textbf{b}$ and denote it by $\Delta:=K_{\textbf{b}} (I(\mathcal{K}_n))$. For the first part of the claim, it suffices to show that the only minimal non-faces of $\Delta$ are $(n-2)$-dimensional.

Let  $F=\{x_1,\ldots, x_n\}$ and $F_1 = F \setminus \{x_1\}$.  Since $x^{\textbf{b}-F } = \prod_{i=2}^n x_i^{w_i-1}$  and $x^{\textbf{b}-F_1 } = x_1\prod_{i=2}^n x_i^{w_i-1}$  are not contained in $I(\mathcal{K}_n)$,  neither $F$ nor $F_1$ is a face in $\Delta$.    In addition, let $F_{i,j} : = F \setminus \{x_i,x_j\}$ for each $1\leq i<j \leq n$. Observe that each $F_{i,j}$ is a face in $ \Delta$  because   $$\displaystyle \frac{x^{\textbf{b}}}{ x^{F_{i,j}}} = \Big( \prod_{k\neq i,j} x_k^{w_k-1}\Big) x_i^{w_i} x_j^{w_j} \in I(\mathcal{K}_n).$$
Therefore, the upper Koszul simplicial complex $\Delta$ has at least one minimal non-face of dimension $(n-2)$. Then, we have $\dim_k \widetilde{H}_{n-3} (\Delta ;k)) \neq 0.$ Hence, by \Cref{thm:bettiKoszul},
$$\beta_{n-1,\textbf{b}} (R/I(\mathcal{K}_n))= \beta_{n-2,\textbf{b}} (I(\mathcal{K}_n))= \dim_k \widetilde{H}_{n-3} (\Delta ;k)) \neq 0.$$
 Note that $\mathcal{K}_n$ does not belong to the class of graphs expressed in \Cref{thm:completeMaxPdim} and it follows that $\pdim (R/I(\mathcal{K}_n)) \leq n-1$. Thus, $\beta_{n-1,|\textbf{b}|} (R/I(\mathcal{K}_n))$
is the unique extremal Betti number of $R/I(\mathcal{K}_n)$ by \Cref{rem:extremalBetti}. 

\end{proof}

\begin{remark}
Recall that, when $w_i \geq 2$ for all $x_i \neq x_1,$ the edge ideal $I(\mathcal{K}_n)$ is the initial ideal of the vanishing ideal of a projective nested cartesian set. As mentioned in the introduction, $\reg (R/I(\mathcal{K}_n))$ is a strict upper bound for the degree of the evaluation map used in creating projective nested cartesian codes.  It was shown in \cite{carvalho2017projective} that degree of the evaluation map must be less than $\sum_{i=1}^n w_i -n+1$ for a projective nested cartesian code to have an ``optimal" minimum distance (\cite[Theorem 3.8]{carvalho2017projective}). However, this upper bound is not obtained by computing $\reg (R/I(\mathcal{K}_n))$ explicitly in \cite{carvalho2017projective} and the equality $\reg (R/I(\mathcal{K}_n)) = \sum_{i=1}^n w_i -n+1$ is rather concluded from their results (see \cite[Proposition 6.3]{martinez2017minimum}). In a way, \Cref{thm:naturalWOC} part (b) recovers this embedded conclusion.
\end{remark}

Above theorem completes the discussion of regularity and projective dimension of $I(\mathcal{K}_n)$ when $\mathcal{K}_n$ is naturally oriented. If the orientation of a weighted complete graph is not known, finding regularity and projective dimension through upper-Koszul simplicial complexes becomes a more difficult task. In the absence of an explicit orientation, structure of the upper-Koszul simplicial complex of $I(\mathcal{K}_n)$ is more  complex. Thus, one needs to employ different techniques than upper-Koszul simplicial complexes.

In the following result, we provide a recursive formula for the  Betti numbers of $I(\mathcal{K}_n)$ in the existence of a sink vertex. This condition  can be considered as a local property of $\mathcal{K}_n$.  Note that if there exists at least one sink vertex, by relabeling the vertices, we may assume that $x_n$ is a sink.

\begin{theorem}\label{thm:completeSink}
If $x_n$ is a sink in $\mathcal{K}_n$, we have
$$ \beta_{i,j} (R/I(\mathcal{K}_n)) =
\begin{cases}
\beta_{i,j} (R/I(\mathcal{K}_{n-1})) +{{n-1}\choose{i}}  + \beta_{i-1,j-w_n} (R/ I(\mathcal{K}_{n-1})) &: j=i+w_n\\\\
\beta_{i,j} (R/I(\mathcal{K}_{n-1})) + \beta_{i-1,j-w_n} (R/ I(\mathcal{K}_{n-1})) &: j\neq i+w_n
\end{cases}$$
for all $i>1.$

\end{theorem}

\begin{proof} Let $J= (x_ix_n^{w_n} : 1\leq i <n)$. Then one can decompose $I(\mathcal{K}_n)$ as a disjoint sum of $I(\mathcal{K}_n)=I(\mathcal{K}_{n-1})+J$ where $\mathcal{K}_{n-1}$ is a weighted oriented complete graph on $(n-1)$ vertices.  It is clear that $\mathcal{K}_{n-1}$ is an induced weighted oriented subgraph of $ \mathcal{K}_n$.  Note that the minimal free resolution of $J$ is obtained from shifting the minimal free resolution of $R/(x_1,\ldots, x_{n-1})$ by degree $w_n$. Thus $J$ has a linear resolution. It then follows from \Cref{thm:bettiSplit} that $I(\mathcal{K}_n)=I(\mathcal{K}_{n-1})+J$ is a Betti splitting because $\mathcal{G}(J)$ contains all the generators of $I(\mathcal{K}_n)$ divisible by $x_n$.  Then, by \Cref{def:bettiSplit},
\begin{align*}
\beta_{i,j} (R/I(\mathcal{K}_n)) &= \beta_{i,j} (R/I(\mathcal{K}_{n-1})) +\beta_{i,j} (R/J)+ \beta_{i-1,j} (R/ I(\mathcal{K}_{n-1}) \cap J)   
\end{align*}
for $i>1$.   Our goal is to analyze each term of the above expression.

It is immediate from the definition of $J$ that $\beta_{i,j+w_n} (R/J)= \beta_{i,j} (R/(x_1,\ldots, x_{n-1}))$ for $i>0.$ Recall that Koszul complex is a minimal free resolution of the $R$-module $R/(x_1,\ldots, x_{n-1})$ and the only non-zero Betti numbers occur when $j=i.$ In particular, for $1\leq i\leq n-1,$ we have
\begin{equation}\label{eq:gen1}
 \beta_{i,i} (R/(x_1,\ldots, x_{n-1})) = {{n-1}\choose{i}}= \beta_{i,i+w} (R/J)   
\end{equation}

as the only non-zero Betti numbers of $J$.

Next, observe that $I(\mathcal{K}_{n-1})\cap J = (x_n^{w_n}) I(\mathcal{K}_{n-1}).$ Then, we can express the Betti numbers of the intersection in terms of iterated Betti numbers of $I(\mathcal{K}_{n-1})$. More specifically, for all $i>0,$
\begin{equation}\label{eq:gen2}
\beta_{i,j+w} (R/I(\mathcal{K}_{n-1}) \cap J ) = \beta_{i,j} (R/I(\mathcal{K}_{n-1})).    
\end{equation}
Therefore, one can obtain the expressions given in the statement of the theorem by using \Cref{eq:gen1} and \Cref{eq:gen2}.
\end{proof}

\begin{corollary}\label{cor:completeSink}
If $x_n$ is a sink in $\mathcal{K}_n$, then
\begin{enumerate}
     \item[(a)] $\pdim (R/I(\mathcal{K}_n))\in  \{ n-1, n\}$ and
     \item[(b)] $ \reg (R/I(\mathcal{K}_n))= \reg (R/I(\mathcal{K}_{n-1}))+(w_n-1).$
\end{enumerate}
\end{corollary}

\begin{proof}
Let $p=\pdim (R/ I(\mathcal{K}_{n-1}))$ and $ r=\reg (R/ I(\mathcal{K}_{n-1})))$.  Then $\beta_{p,p+r} (R/ I(\mathcal{K}_{n-1})) \neq 0$ and $\beta_{i,j} (R/ I(\mathcal{K}_{n-1})) =0$ for $i>p$ or $j>p+r.$ By using  \Cref{thm:completeSink}, we obtain the  following top non-zero Betti numbers of $R/I(\mathcal{K}_n)$.
\begin{align}
    \beta_{p+1,p+r+w_n} (R/I(\mathcal{K}_n)) & \neq 0 \label{eq:complete1}\\
    \beta_{n-1,n-1+w_n} (R/ I( \mathcal{K}_n)) & \neq 0 \label{eq:complete2}
\end{align}

Hence, \Cref{eq:complete1} and \Cref{eq:complete2} imply that
\begin{align*}
    \reg (R/I(\mathcal{K}_n))& = \max \{ w_n, r+w_n-1\} =r+w_n-1\\
    &= \reg (R/I(\mathcal{K}_{n-1}))+(w_n-1).
\end{align*}
Similarly, using the top non-zero Betti numbers, we have
$$    \pdim (R/I(\mathcal{K}_n)) = \max \{ n-1, p+1\}$$ 
Since $p \leq n-1,$ the projective dimension is either $n-1$ or $n.$

\end{proof}

\begin{remark}\label{rem:star}
If the underlying graph of a weighted oriented graph $\DD$ on $n$ vertices is a star, we call $\DD$ a weighted oriented star graph.  Let $x_n$ be the center of $\DD$. If $x_n$ is a sink vertex, we say $\DD$ is a weighted oriented star with a center sink. The edge ideal of a weighted oriented star with a center sink $x_n$ is given as
$$I(\DD) = (x_ix_n^{w_n} ~:~ 1\leq i\leq n-1).$$
As discussed in the proof of \Cref{thm:completeSink}, the module $R/I(\DD)$ has a linear resolution and it is obtained by shifting the Koszul complex of $R/(x_1,\ldots, x_{n-1})$ by degree $w_n$. Then $$\pdim (R/I(\DD)) =n-1 \text{ and} \reg (R/I(\DD)) =w_n.$$

\end{remark}

\section{Betti Numbers via Mapping Cone Construction}\label{ch5}

In this section, we provide a recursive formula for the Betti numbers of edge ideals of weighted oriented graphs with at least one leaf vertex which is also a sink. We achieve it by employing a technique called the \emph{mapping cone construction}. This technique is different than Betti splittings while being as powerful.  

Recall that Betti splitting is a method which allows one to express Betti numbers of an ideal in terms of smaller ideals. In a similar vein, mapping cone construction allows one to build a free resolution of an $R$-module $M$ in terms of $R$-modules associated to $M$. In particular, given a short exact sequence
$$0 \longrightarrow R/M' \longrightarrow R/M'' \longrightarrow R/M \longrightarrow 0$$
where $M', M''$ and $M$ are graded $R$-modules, the mapping cone construction provides a free resolution of $M$ in terms of free resolutions of $M'$ and $M''$. For more details on the mapping cone construction, we refer the reader to \cite{Peeva2011}. In general, given minimal free resolutions for $M'$ and $M''$, the mapping cone construction does not necessarily give a minimal free resolution of $M$. However, there are classes of ideals in which the mapping cone construction provides a minimal free resolution for particular short exact sequences (see \cite{bouchat2011path}).

Let $\DD$ be a weighted oriented graph with the vertex set $V(\DD)= \{x_1,\ldots, x_n\} $. A vertex is called a \emph{leaf} if there is only one edge incident to it. In the existence of at least one leaf vertex which is also a sink, one can use the mapping cone construction to describe Betti numbers of $R/I(\DD)$ recursively.  Note that there is no restriction on the overall orientation of $\DD$.

\begin{theorem}\label{thm:mappingCone}
Let $\DD$ be a weighted oriented graph on the vertices $x_1,\ldots, x_n$ with a leaf $x_n$. Suppose $x_n$ is a sink vertex. Then the mapping cone construction applied to the short exact sequence
$$0 \longrightarrow \frac{R}{I(\DD \setminus x_n ): (x_{n-1}x_n^{w_n})} (-w_n-1) \xrightarrow{x_{n-1}x_n^{w_n}} \frac{R}{I(\DD \setminus x_n )} \longrightarrow \frac{R}{I(\DD)} \longrightarrow 0$$

provides a minimal free resolution of $R/I(\DD)$. In particular, for any $i$ and $j$, we have

$$\beta_{i,j} (R/I(\DD)) = \beta_{i,j} (R/I(\DD \setminus x_n )) +\beta_{i-1,j-w-1} (R/I(\DD \setminus x_n ):x_{n-1})). $$
\end{theorem}

\begin{proof}
Let $\DD'$ denote the  weighted oriented induced subgraph $\DD \setminus x_n$ of $\DD$ and let $x_{n-1}$ be the unique neighbor of $x_n$ such that $(x_{n-1},x_n)\in E(\DD)$. Since $x_n^{w_n}$ does not divide a minimal generator of $I(\DD')$, one has
$$I(\DD'):(x_{n-1}x_n^{w_n}) = I(\DD'):x_{n-1}.$$
Then, it implies that the exact sequence

\begin{equation}\label{eq:1}
    0 \longrightarrow \frac{R}{I(\DD'): (x_{n-1}x_n^{w_n}))} (-w_n-1) \xrightarrow[\delta]{x_{n-1}x_n^{w_n}} \frac{R}{I(\DD')} \longrightarrow \frac{R}{I(\DD)} \longrightarrow 0
\end{equation}
factors as

\begin{equation}\label{eq:2}
   \begin{tikzcd}
    0\arrow[r] & \frac{R}{I(\DD'): (x_{n-1}x_n^{w_n})} (-w_n-1)\arrow[r, "x_{n-1}x_n^{w_n}"] \arrow[d,  "x_n^{w_n}"]&  \frac{R}{I(\DD')}  \arrow[r]& \frac{R}{I(\DD)}\arrow[r]&  0. \\
 & R/ I(\DD') :x_{n-1} \arrow[ru,  "x_{n-1}"] &      &  & 
\end{tikzcd}  
\end{equation}

Let 
\begin{align*}
 \mathcal{F}&: ~~0 \cdots \xrightarrow{\phi_2} F_1 \xrightarrow{\phi_1} F_0=R \xrightarrow{\phi_0}  R/I(\DD'): (x_{n-1}x_n^{w_n}) \longrightarrow 0 \text{ and }\\
  \mathcal{G}&: ~~0 \cdots \xrightarrow{\psi_2} G_1 \xrightarrow{\psi_1} G_0=R \xrightarrow{\psi_0} R/I(\DD') \longrightarrow 0 
\end{align*}
be minimal free resolutions of $R/I(\DD'): (x_{n-1}x_n^{w_n})$ and $R/I(\DD'),$ respectively.  Then the mapping construction applied to \Cref{eq:1} provides a free resolution of $R/I(\DD)$ given by
$$0 \cdots \xrightarrow{\varphi_3} G_2\oplus F_1(-w_n-1) \xrightarrow{\varphi_2} G_1\oplus R(-w_n-1) \xrightarrow{\varphi_1} R \xrightarrow{\varphi_0} R/I(\DD) \longrightarrow 0  $$
where the map $\varphi_i$'s are defined by $\varphi_1 = \begin{array}{cc}
[\psi_1 & -\delta_0] 
\end{array}$ and
$$\varphi_i = 
\left[ \begin{array}{cc} \psi_i & (-1)^i\delta_{i-1}\\ 0 & \phi_{i-1}\end{array}\right]$$
for $i>1$ such that each $\delta_i: F_i(-w_n-1) \longrightarrow G_i$ is induced from the homomorphism $\delta$.

It follows from the factorization in \Cref{eq:2} that the entries of the matrix of $\delta_i$ are not units. Since $\mathcal{F}$ and $\mathcal{G}$ are minimal free resolutions, then none of the entries in the  matrix representation of $\varphi_i$ can be units. Thus the mapping cone construction applied to \Cref{eq:1} results with a minimal free resolution of $R/I(\DD)$. In particular, this implies the following recursive formula for the Betti numbers of $R/I(\DD)$
$$\beta_{i,j} (R/I(\DD)) = \beta_{i,j} (R/I(\DD')) +\beta_{i-1,j-w_n-1} (R/I(\DD'):x_{n-1})$$
for any $i, j$.
\end{proof}

\begin{corollary}\label{cor:mappingCone}
Let $\DD$ be a weighted oriented graph on the vertices $x_1,\ldots, x_n$ such that $x_n$ is a leaf  and a sink vertex.  Then
\begin{enumerate}
    \item[(a)] $\reg (R/I(\DD)) = \max \{ \reg (R/I(\DD \setminus x_n)), \reg (R/I(\DD \setminus x_n):x_{n-1})+1 \}$ and\\
        \item[(b)]  $\pdim (R/I(\DD)) = \max \{ \pdim(R/I(\DD \setminus x_n)), \pdim (R/I(\DD \setminus x_n):x_{n-1})+1 \}.$
\end{enumerate}
\end{corollary}

\subsection{Application} Let $\mathcal{P}_n$ denote a  weighted naturally oriented path on $n$ vertices. If all non-source vertices have non-trivial weights, regularity and projective dimension formulas follow from \Cref{cor:weightedForest}. If one allows non-source vertices to have trivial weights, computing the regularity and the projective dimension becomes a much more complicated task as these invariants heavily rely on the orientation of the graph and the positions of non-trivial weights. Providing formulas for the regularity and projective dimension of any weighted oriented graph is an open problem.

In an attempt to address this general problem, weighted naturally oriented paths and cycles are studied in \cite{BBKO}. Indeed, positions of non-trivial weights is quite crucial in computing the regularity (see \cite[Theorem 5.9]{BBKO}). Particularly, whenever there are consecutive non-trivial weight vertices  $x_i$ and $x_{i+2}$ such that $x_{i+1}$ has a trivial weight, then $x_i$ and $x_{i+2}$ can not ``contribute" to the regularity simultaneously. One needs to consider the contribution of the one or the other and determine the regularity by taking the maximums of corresponding contributions (see \cite[Notation 5.4, Definition 5.6 and Theorem 5.9]{BBKO}).

In what follows, we consider a more general case than that of \cite{BBKO} and provide a recursive formula for the Betti numbers of egde ideal of a weighted oriented path. Furthermore, our recursive formulas can offer an explanation about the ``distance two away condition" of \cite[Theorem 5.9]{BBKO}.

\begin{corollary}\label{cor:recursiveBettiPath}
Let $\mathcal{P}_n$ be a  weighted oriented path on the vertices $x_1,\ldots, x_n$ such that $(x_{n-2},x_{n-1}),$ $ (x_{n-1},x_n) \in E(\mathcal{P}_n)$.
\begin{enumerate}
    \item[(a)]  If $x_{n-1}$ is a non-trivial vertex, then
    $$\reg (R/I(\mathcal{P}_n)) = \reg (R/I(\mathcal{P}_{n-1})) +w_n-1.$$
    \item[(b)] If $x_{n-1} $ is a trivial vertex, then 
    $$\reg (R/I(\mathcal{P}_n)) = \max \{ \reg (R/I(\mathcal{P}_{n-1})), \reg (R/I(\mathcal{P}_{n-3}))+w_n \}$$
\end{enumerate}

\end{corollary}
\begin{proof}
(a) Suppose $w_{n-1} >1.$ Then $I(\mathcal{P}_{n-1}):x_{n-1} =I(\mathcal{P}'_{n-1}) $  where  $\mathcal{P}'_{n-1}$ is a weighted reduced form of $\mathcal{P}_{n-1}$ on  $x_{n-1}$. Thus, it follows from \Cref{cor:mappingCone}  that
\begin{equation}\label{eq:pathGen}
  \reg (R/I(\mathcal{P}_n)) = \max \{ \reg (R/I(\mathcal{P}_{n-1})), \reg (R/I(\mathcal{P}'_{n-1})) +w_n\}.  
\end{equation}

Since $x_{n-1}$ is a sink vertex with a non-trivial weight in $\mathcal{P}'_{n-1}$, we have $\reg (R/I(\mathcal{P}_{n-1}))= \reg (R/I(\mathcal{P}'_{n-1}))+1.$ 

By  making use of \Cref{cor:sinkReduction} and the fact that $w_n \geq 1$, \Cref{eq:pathGen} yields to the following
$$\reg (R/I(\mathcal{P}_n))  = \reg (R/I(\mathcal{P}'_{n-1})) +w_n = (\reg (R/I(\mathcal{P}_{n-1})) -1 ) +w_n.$$

(b) Suppose $w_{n-1}=1.$ Let $I':= I(\mathcal{P}_{n-1}):x_{n-1} = I(\mathcal{P}_{n-3})+(x_{n-2})$. Since $x_{n-2}$ does not divide any minimal generator of $I(\mathcal{P}_{n-3}),$ one can obtain the minimal free resolution of $R/ I'$ by taking the tensor product of minimal free resolutions of  $R/ I(\mathcal{P}_{n-3})$ and $R/(x_{n-2})$. Then
$$\reg ( R/I' ) = \reg  ( R/ I(\mathcal{P}_{n-3}) ),$$
and the statement follows from \Cref{cor:mappingCone}
\end{proof}

\begin{remark}
Let $\mathcal{P}_n$ be a  weighted naturally oriented path on the vertices $x_1,\ldots, x_n$. If $w_n >1$ and $w_{n-1}=1$, we can use \Cref{cor:recursiveBettiPath}  part (b) to determine the regularity of $R/I(\mathcal{P}_n)$ inductively by taking the maximum of the following two expressions.
 $$\reg (R/I(\mathcal{P}_n)) = \max \{ \reg (R/I(\mathcal{P}_{n-1})), \reg (R/I(\mathcal{P}_{n-3}))+w_n \}$$
Note that the ideal in the first expression contains $x_{n-2}$ in its support. However, vertex  $x_n$ is not in the support of the first ideal and its weight does not contribute to the regularity in the first expression. On the other hand, the second expression contains $w_n$, the weight contribution of $x_n$, and the ideal associated to it does not contain $x_{n-2}$ in its support. Thus, \Cref{cor:recursiveBettiPath}  part (b) exhibits the behavior of distance two away non-trivial weights in the regularity computations.
\end{remark}

\section{Questions}\label{ch6}

\begin{question}
Let $\DD$ be a weighted oriented graph and $G$ be its underlying graph on $n$ vertices.
\begin{enumerate}
 \item[(a)]  Is there any relation between the Betti numbers of $R/I(G)$ and $R/I(\DD)$?
    \item[(b)]   Is  $\pdim (R/I(G)) \leq \pdim (R/I(\DD))$?
       \item[(b)]  Is $\reg (R/I(G)) \leq \reg (R/I(\DD))$?
\end{enumerate}
\end{question}

Intuition and computational evidence suggests that both questions have positive answers.

Our next question is motivated by \Cref{cor:Betti} and \Cref{cor:sinkReduction}. In these two corollaries we provide a positive answer to the following questions when $x_i$ is a non-trivial sink vertex. It is natural to wonder whether it is true for any non-trivial vertex $x_i$.

\begin{question}\label{q:2}
Let $\DD$ be a weighted oriented graph with a non-trivial weight vertex $x_i$  and let $\DD'$ be a weight reduced form of $\DD$ on $x_i.$ 
\begin{enumerate}
\item[(a)] When is $\beta_i (R/I(\DD)) = \beta_i (R/I(\DD')) $  for all $i \geq 0$?
\item[(b)] Is there any relation between $\beta_{i,j} (R/I(\DD))$ and $ \beta_{i,j} (R/I(\DD')) $?
\item[(c)] When is $\pdim (R/I(\DD)) = \pdim (R/I(\DD'))$?
\item[(d)] When is $\reg (R/I(\DD)) = \reg (R/I(\DD')) +1$?
\end{enumerate}
\end{question}

\begin{example}

Let $I(\DD)= (x_2x_1,x_3x_2^3,x_4x_3^2,x_4x_5)$. Consider the following chain of weight reductions where $\DD'$ is a weight reduced form of $\DD$ on $x_2$,  $\DD''$  is a weight reduced form of $\DD'$ on $x_3$, and $\DD'''$ is a weight reduced form of $\DD''$ on $x_2$ with the corresponding edge ideals given as
\begin{align*}
   I(\DD') &=  (x_2x_1,x_3x_2^2,x_4x_3^2,x_4x_5)\\
      I(\DD'') &= (x_2x_1,x_3x_2^2,x_4x_3,x_4x_5) \\
      I(\DD''') & =(x_2x_1,x_3x_2,x_4x_3,x_4x_5).
\end{align*}

Below, we present the Betti tables of $I(\DD), I(\DD'), I(\DD''),$ and $I(\DD''')$, in order. 
	\medskip 
	{
\begin{verbatim}
                 0   1   2   3   4                0   1   2   3   4         
            -----------------------          --------------------------  
            0:   1   -   -   -   -           0:   1   -   -   -   -      
            1:   -   2   -   -   -           1:   -   2   -   -   -     
            2:   -   1   2   -   -           2:   -   2   -   -   -   
            3:   -   1   2   1   -           3:   -   -   3   4   1 
            4:   -   -   2   3   1             
            -----------------------          --------------------------
            Tot: 1   4   6   4   1           Tot:  1   4   6   4   1
		\end{verbatim}
	}
\medskip 
	\medskip 
	{
\begin{verbatim}
                 0   1   2   3                     0   1   2   3  
            -----------------------          ----------------------
            0:   1   -   -   -               0:    1   -   -   -   
            1:   -   3   1   -               1:    -   4   3   -          
            2:   -   1   4   2               2:    -   -   1   1 
            -----------------------          ---------------------- 
            Tot: 1   4   5   2               Tot:  1   4   4   1 
   
		\end{verbatim}
	}
\medskip 	

Based on the above Betti tables, equalities in \Cref{q:2} (a),(c),(d)
 hold for $\DD$ and $\DD'$. However, we have 
 $$\pdim (R/I(\DD')) = \pdim (R/I(\DD''))+1,$$
$$\reg (R/I(\DD'')) = \reg(R/I(\DD''')),$$
indicating that suggested equalities in \Cref{q:2} are not always valid.  Computational experiments suggest that the desired equalities hold for $\DD$ and a reduced form of $\DD$ on $x_i$ where $w_i >2$.
\end{example}

Answering above questions can help towards improving our understanding on the behavior of Betti numbers  of monomial ideals under certain monomial operations.

\bibliographystyle{plain}
\bibliography{references}

\end{document}